\documentclass[reqno,12pt]{amsart}
\usepackage[T1]{fontenc}
\usepackage[latin1]{inputenc}
\usepackage{setspace}
\usepackage{epsfig}
\usepackage{setspace}
\makeatletter

 \theoremstyle{plain}
 \newtheorem{theorem}{Theorem}[section]
 \newtheorem{Def}[theorem]{Definition}
 \newtheorem{lemma}[theorem]{Lemma}
 \newtheorem{corollary}[theorem]{Corollary}
 \numberwithin{equation}{section} 
\numberwithin{figure}{section} 

 \newtheorem{remark}[theorem]{Remark}
 \newtheorem*{acknowledgement*}{Acknowledgement}

\usepackage[T1]{fontenc}
\usepackage[latin1]{inputenc}
\usepackage{amssymb}

\newcommand{\R}{\mathbb R}
\newcommand{\C}{\mathcal C}

\newcommand{\N}{\mathbb N}

\newcommand{\Span}{\textnormal{span}}


\title[Boundary control of elliptic solutions]{Boundary control of elliptic solutions to enforce local constraints}

\author{G. Bal$^1$ and M. Courdurier$^2$}
\thanks{$^1$ Department of Applied Physics and Applied Mathematics, Columbia University,
New York, NY 10027; gb2030@columbia.edu}
\thanks{$^2$ Departamento de Matem\'aticas, Pontificia Universidad Cat\'olica de Chile, Santiago, Chile; mcourdurier@mat.puc.cl}

\begin{document}
\maketitle

\begin{abstract}

We present a constructive method to devise boundary conditions for solutions of second-order elliptic equations so that these solutions satisfy specific qualitative properties such as: (i) the norm of the gradient of one solution is bounded from below by a positive constant in the vicinity of a finite number of prescribed points; and (ii) the determinant of gradients of $n$ solutions is bounded from below in the vicinity of a finite number of prescribed points. Such constructions find applications in recent hybrid medical imaging modalities.

The methodology is based on starting from a controlled setting in which the constraints are satisfied and continuously modifying the coefficients in the second-order elliptic equation. The boundary condition is evolved by solving an ordinary differential equation (ODE) defined so that appropriate optimality conditions are satisfied. Unique continuations and standard regularity results for elliptic equations are used to show that the ODE admits a solution for sufficiently long times.

\end{abstract}

\section{Introduction}
\label{sec:intro}

Several recent hybrid medical imaging modalities may be recast as systems of nonlinear partial differential equations with known sources; see, e.g., \cite{A-Sp-08,AS-IP-12,B-Irvine-12,B-IO-12,KK-EJAM-08,S-SP-2011,WW-W-07} for reference on such modalities. The solution of such systems requires that said sources satisfy specific properties which may often be recast as specific, qualitative properties of solutions of second-order partial differential equations. In the applications presented in, e.g., \cite{BR-IP-11,BU-IP-10}, solutions of second-order elliptic equations are required to have gradients that do not vanish, at least locally. In other applications described in, e.g., \cite{ABCTF-SIAP-08,BBMT-12,CFGK-SJIS-09,MB-IPI-12,MB-IP-12}, the determinant of the gradients of $n$ solutions in spatial dimension $n$ is required to be bounded away from $0$. 

Such qualitative properties are to be ensured by controlling the boundary conditions of the elliptic solutions. Using theories based on complex geometric optics solutions or on unique continuation principles and Runge approximations, it is shown in, e.g., \cite{BU-CPAM-12,BU-IP-10,T-IP-10} that the qualitative properties are satisfied for an open set of boundary conditions that is not precisely characterized. 

This paper presents a methodology to construct boundary conditions such that the qualitative properties are satisfied locally. To simplify the presentation, we consider the setting of a second-order elliptic equation in divergence form with an arbitrary (elliptic) diffusion coefficient. Starting from a configuration where the diffusion coefficient is constant and where boundary conditions can easily be defined so that the qualitative property is satisfied, we propose to continuously deform the diffusion coefficient from the constant one to the final coefficient of interest. An ordinary differential equation (ODE) is then prescribed for the evolution of the boundary condition so that the qualitative property of interest is satisfied, at least locally in the vicinity of a finite number of points of interest, during the whole homotopy transformation. The qualitative property is recast as an adapted set of constraints. The ODE is tailored so that optimality conditions are met to satisfy the set of constraints. That the ODE solution exists for the whole duration of the homotopy transformation is guaranteed by using a unique continuation principle for solutions to elliptic equations. The whole procedure may be seen as an optimal boundary control method so that the elliptic solutions satisfy appropriate constraints inside the domain.

The rest of the paper is structured as follows. The construction of boundary conditions ensuring that the gradient of the solution does not vanish in the vicinity of a given point is introduced in section \ref{sec:description}.  Section \ref{Main} presents the main results of this paper.
In Section \ref{ConstructingF}, we describe the optimality conditions that justify our choice of the evolution equation (the ODE) and give an example of a simpler, more naive, construction that does not achieve our objectives. 
Section \ref{Proofs} contains the proofs of the main results. 
Section \ref{Extensions} generalizes the construction to other settings including the construction of solutions such that the gradients do not vanish at a finite number of points and the construction of solutions whose gradients form a basis in the vicinity of a finite number of points.

\section{Description of the Problem and Formulation of the Method}\label{Formulation}
\label{sec:description}

Let $X$ be a bounded domain in $\R^n$ with boundary $\partial X$.
For a given coefficient $\gamma(x)$ and a fixed $\hat{x}\in X$, the
goal is to find a boundary condition $\hat{f}$ such that
$|\nabla u(\hat{x})|\geq1$, where $u$ is the solution of the equation
\begin{align*}
\begin{cases}\nabla\cdot(\gamma\nabla u)=0 &\textnormal{ in } X\\
u=\hat{f} &\textnormal{ in } \partial X.
\end{cases}
\end{align*}

In order to construct such an $\hat{f}$ we propose an evolution scheme. Namely,
for a given $\gamma_0(x)$ let $\gamma_s:=(1-s)\gamma_0+s\gamma, \forall s\in[0,1]$.
For a family of boundary conditions $\{f_s\}_{s\in[0,1]}$ let $u_s$ denote the
corresponding solution of 
\begin{align*}
(P_s) \begin{cases}\nabla\cdot(\gamma_s\nabla u_s)=0 &\textnormal{ in } X\\
u_s=f_s &\textnormal{ in } \partial X.
\end{cases}
\end{align*}
The proposed scheme consists in constructing $\{f_s\}_{s\in[0,1]}$ with the property that $|\nabla u_s(\hat{x})|$
is non-decreasing after choosing  $\gamma_0$ and $f_0$ such that
$|\nabla u_0(\hat{x})|\geq1$; for example $\gamma_0\equiv1$ and $f_0(x_1,x_2,...,x_n)=x_1$.
Thus, $\hat{f}:=f_1$ is a solution of the problem since $|\nabla u_1(\hat{x})|\geq1$.\\

To construct $\{f_s\}_{s\in[0,1]}$,
let us assume that $f_s=f_0+\int_0^s g_t dt$ and denote $u_s'=\partial u_s/\partial s$
and $\gamma_s'=\partial \gamma_s/\partial s=\gamma-\gamma_0$. Differentiating $(P_s)$
with respect to $s$ gives the equation
\begin{align*}
\begin{cases}\nabla\cdot(\gamma_s\nabla u_s')+\nabla\cdot(\gamma_s'\nabla u_s)=0 &\textnormal{ in } X\\
u_s'=g_s &\textnormal{ in } \partial X.
\end{cases}
\end{align*} 
The condition that  $|\nabla u_s(\hat{x})|$ is non-decreasing becomes
$\nabla u_s(\hat{x})\cdot\nabla u_s'(\hat{x})\geq0$.
Such a characterization hints at the construction
of $\{f_s\}_{s\in[0,1]}$ by means of an initial value problem.

We construct $\{f_s,g_s: s\in [0,1]\}$ as the solution of
\begin{align*}
\begin{cases} g_s=\frac{\partial f_s}{\partial s}=F(f_s,s)\\
f_s\Big|_{s=0}=f_0
\end{cases}
\end{align*}
for an $F$ satisfying two specific conditions.

The first condition on the functional $F$ is that it guarantees
$\nabla u_s(\hat{x})\cdot\nabla u'_s(\hat{x})\geq0$. The second condition on $F$ is that it admits a solution
for initial value problem, with initial condition $f_0$, for all $s\in[0,1]$.

In this work, we provide an explicit description of a functional $F$ (Definition \ref{definitionF})
satisfying those two conditions (Theorems \ref{firstpropF}, \ref{secondpropF}), hence not only
solving the original problem, but also providing an explicit method to construct the solution $\hat{f}$.\\

\section{Notation, Framework and Main Results}\label{Main}

\subsection{Notation} The following notation will be used through the paper.
Let $X$ be a bounded domain, let $\partial X$ denote its boundary,
at $x\in\partial X$ let $\nu(x)$ denote the outer unit normal to $X$. Let $\overline{X}$ be the closure of $X$.
The notation $C^{k,\alpha},k\in\N, 0<\alpha\leq1$ represents H\"older continuity, i.e., $k$ continuous
derivatives with the $k$-th derivative being H\"older continuous of order $\alpha$; in the case $\alpha$=1
the k-th derivative is Lipschitz continuous. Let $C^{k,\alpha}(\Omega)$ be the space of H\"older
continuous functions from $\Omega$ into $\R$ and write $X\in C^{k,\alpha}$
to mean that $\partial X$ can be locally represented as the graph of a H\"older continuous function.
In $C^{k,\alpha}(\Omega)$ the norm of a function $f$ is written as $|f|_{k,\alpha,\Omega}$.
We use the classical notation for the integrable spaces $L^p(\Omega)$ and the norm $|f|_{L^p(\Omega)}, p\in[1,\infty]$.
Denote as $W^{k,p}(\Omega)$, $k\in\N, p\in[1,\infty]$, the Sobolev space of functions with
$k$ weak derivatives in $L^p(\Omega)$. In these spaces, we consider the usual norm that makes
them Banach spaces. Let $W_0^{k,p}(\Omega)$ be the completion of $C^\infty_0(\Omega)$
in $W^{k,p}(\Omega)$; see \cite{gt1} for additional details.

\subsection{Hypotheses} The following hypotheses will be assumed throughout this section.
We assume that $X$ is a bounded subset in $\R^n$, $n\in\N$ fixed. We fix $p\in(1,\frac n{n-1})$ and let
$\alpha=n\frac {p-1}{p}\in (0,1)$. We fix $k\in \N$.\\

We assume that $X$ is a $C^{k+3,\alpha}$ bounded domain, $\hat{x}\in X$ is fixed.\\

We assume that $\gamma\in C^{k+n+3}(\overline{X})$ and that there exist constants $c,C$ such that
$0<c<\gamma<C$ in $\overline{X}$.\\

Let  $\gamma_0\equiv1$ and $f_0(x_1,x_2,...,x_n)=x_1$. For $s\in [0,1]$
define $\gamma_s:=(1-s)\gamma_0+s\gamma$.

\subsection{Main Results} The first theorem summarizes classical results and shows that the formal
calculations in Section \ref{Formulation} are valid in this setting.

\begin{theorem}\label{frame} Let $s\mapsto f_s \in C^1\big([0,1]; C^{k+2,\alpha}(\partial X)\big)$. For each $s\in[0,1]$
there is a unique solution $u_s\in C^{k+2,\alpha}(\overline{X})$ of the equation 
\begin{align*}
(P_s) \begin{cases}\nabla\cdot(\gamma_s\nabla u_s)=0 &\textnormal{ in } X\\
u_s=f_s &\textnormal{ in } \partial X
\end{cases}
\end{align*}
and $s\mapsto u_s\in  C^1\big([0,1]; C^{k+2,\alpha}(\overline{X})\big)$. Let
$u_s'=\partial u_s/\partial s, f_s'=\partial u_s/\partial s$ and
$\gamma_s'=\partial \gamma_s/\partial s=\gamma-\gamma_0$. Then $u_s'$ satisfies
the equation
\begin{align*}
(P'_s) \begin{cases}\nabla\cdot(\gamma_s\nabla u_s')+\nabla\cdot(\gamma_s'\nabla u_s)=0 &\textnormal{ in } X\\
u_s'=f_s' &\textnormal{ in } \partial X
\end{cases}
\end{align*}
and $\frac{d}{ds}\big(\frac{1}{2}|\nabla u_s(\hat{x})|^2\big)=\nabla u_s(\hat{x})\cdot \nabla u_s'(\hat{x})$.
\end{theorem}

For $y\in \R$ let $\partial_y =(y\cdot \nabla)$ denote the $y$-directional derivative in $\R^n$.
Let $s\in[0,1]$. The following auxiliary problem is crucial in our analysis:
\begin{align*}
(A_s)\begin{cases}\nabla\cdot(\gamma_s\nabla \lambda)=\partial_y\delta_{\hat{x}}&\textnormal{ in } X\\
\lambda=0&\textnormal{ in } \partial X.
\end{cases}
\end{align*}
Here, $\delta_{\hat{x}}$ is the distribution at $\hat x$ such that $\int_{X} \delta_{\hat x} f(x) dx = f(\hat x)$. The dependence of $\lambda$ on $s$ is not written explicitly since it will be clear from the context.

\begin{theorem}\label{existslambda} The problem $(A_s)$ above has a unique solution 
$\lambda\in L^p(X)\cap C^{k+3,\alpha}(\overline{X}\setminus \{\hat{x}\})$.
If $U\subset (\overline{X}\setminus\{\hat{x}\})$ is compact, then
$s\mapsto \lambda|_U \in C([0,1]; L^p(X)\cap C^{k+3,\alpha}(U))$.
\end{theorem}

We proceed to define an adequate functional $F$ for the initial value problem of $\{f_s\}_{s\in[0,1]}$.\\

\begin{Def} \label{definitionF}
Given $f\in C^{k+2,\alpha}(\partial X)$ and $s\in[0,1]$, let $u\in C^{k+2,\alpha}(X)$ be the solution of
\begin{align*}
\begin{cases}\nabla\cdot(\gamma_s\nabla u)=0 &\textnormal{ in } X\\
u=f &\textnormal{ in } \partial X.
\end{cases}
\end{align*}
Let $\lambda$ be the solution of
\begin{align*}
\begin{cases}\nabla\cdot(\gamma_s\nabla \lambda)=\nabla u(\hat{x}) \cdot \nabla \delta_{\hat{x}}&\textnormal{ in } X\\
\lambda=0&\textnormal{ in } \partial X.
\end{cases}
\end{align*}
If $\nabla u(\hat{x}) = 0$ let $\mu>0$, otherwise let
\begin{align*}
\mu= \frac{\int_X\lambda\nabla \cdot((\gamma-\gamma_0)\nabla u)}{\qquad \Big|\gamma_s\frac{\partial\lambda}{\partial\nu}\Big|^2_{L^2(\partial X)}}.
\end{align*}
We define $F:C^{k+2,\alpha}(\partial X)\times[0,1]\to C^{k+2,\alpha}(\partial X)$ as
\begin{align*}
F(f,s):= \begin{cases}
0& \textnormal{ if } \mu\geq 0\\
\mu \gamma_s(\frac{\partial\lambda}{\partial \nu})& \textnormal{ if } \mu\leq 0.
\end{cases}
\end{align*}
\end{Def}
The functional $F$ satisfies the required properties.

\begin{theorem}\label{firstpropF}
Given $f\in C^{k+2,\alpha}(\partial X)$ and $s\in[0,1]$, let $u\in C^{k+2,\alpha}(X)$
be the solution of
\begin{align*}
\begin{cases}\nabla\cdot(\gamma_s\nabla u)=0 &\textnormal{ in } X\\
u=f &\textnormal{ in } \partial X.
\end{cases}
\end{align*}
Let $g=F(f,s)$ and let $v$ be the solution of 
\begin{align*}
\begin{cases}\nabla\cdot(\gamma_s\nabla v)+\nabla\cdot(\gamma_s'\nabla u)=0 &\textnormal{ in } X\\
v=g &\textnormal{ in } \partial X.
\end{cases}
\end{align*}
Then $\nabla u(\hat{x})\cdot \nabla v(\hat{x})\geq 0$.
\end{theorem}

The second property for $F$ requires a strong relationship
between the solution of the auxiliary problem $(A_s)$ and its normal derivative
at the boundary. In particular, the following injectivity results is needed.

\begin{theorem}\label{injectlambda}
Let $\lambda\in L^p(X)\cap C^{k+3,\alpha}(\overline{X}\setminus \{\hat{x}\})$ be the
solution of $(A_s)$ and let $\gamma_s\partial \lambda/\partial \nu\big|_{\partial X}\in C^{k+2,\alpha}(\partial X)$
be its normal derivative at the boundary. Then
\begin{align*}
\Big[\lambda\equiv0\Big]\Leftrightarrow \Big[\gamma_s\frac{\partial\lambda}{\partial \nu}\Big|_{\partial X}\equiv0\Big]\Leftrightarrow \Big[y=0\Big].
\end{align*}
\end{theorem}

This motivates us to regard $\lambda$ and its normal derivative as functions of $y\in\R^n$,
a finite dimensional space. Using the continuous dependence of $\lambda$ on $s$, 
we can recast the previous injectivity Theorem as an apparently stronger result.

\begin{corollary}\label{estlambda}
Let $\lambda$ be the solution of $(A_s)$
and let $\gamma_s\partial \lambda/\partial \nu\big|_{\partial X}$ be its normal derivative at the boundary.
There exists constants $a,b,\rho,\eta>0$ independent of $y\in \R^n$ and independent of $s\in[0,1]$, such that
\begin{align*}
\rho|y|\leq a\Big|\gamma_s\frac{\partial\lambda}{\partial \nu}\Big|_{k+2,\alpha,\partial X}\leq \Big| \lambda\Big|_{L^p(X)}
\leq b \Big|\gamma_s\frac{\partial\lambda}{\partial \nu}\Big|_{L^2(\partial X)} \leq \eta |y|.
\end{align*}
In particular, for any $\eta>0$, the quantities
\begin{align*}
\frac{\lambda}{\qquad \Big|\gamma_s\frac{\partial\lambda}{\partial \nu}\Big|_{L^2(\partial X)}} \in
L^p(X) \quad \textnormal{ and } \quad
\frac{\gamma_s\frac{\partial \lambda}{\partial \nu}}
{\qquad \Big|\gamma_s\frac{\partial\lambda}{\partial \nu}\Big|_{L^2(\partial X)}}\in C^{k+2,\alpha}(\partial X),
\end{align*}
 as functions of $y$, are uniformly Lipschitz in $\{y\in\R^n: |y| \geq \eta\}$, independently of $s\in[0,1]$.
\end{corollary}

We start at $s=0$ with an adequate $f_0,\gamma_0$, hence the estimates of the Corollary \ref{estlambda}
will imply the solvability of the initial value problem for all $s\in[0,1]$.

\begin{theorem}\label{secondpropF}
There exists a unique solution $s\mapsto f_s$ in $C^1\big([0,1];C^{k+2,\alpha}(\partial X)\big)$
of the initial value problem
\begin{align*}
\begin{cases}
\frac{\partial}{\partial s}f_s=F(f_s,s)\\
f_s\big|_{s=0}=f_0.
\end{cases}
\end{align*}
\end{theorem}
In summary, for $F$ as in Definition \ref{definitionF}, the initial value
problem admits a solution for $s\in[0,1]$ and $\hat{f}=f_1$ solves the original problem.
\section{Some Aspects about the Construction of $F$}\label{ConstructingF}

In this section we elaborate on the requirements on
$F:\C^{k+1,\alpha}(\partial X)\times[0,1]\to \C^{k+1,\alpha}(\partial X)$
that lead us to Definition \ref{definitionF}. We start by presenting a simple, naive, and flawed
construction that exemplifies some of the difficulties before proceeding to
the optimal aspects of Definition \ref{definitionF}.\\

Let us consider scalings of the initial boundary condition, namely
we let $f_s=\phi(s)f_0$, where $\phi\in C^1([0,1];\R)$.
Let $u_s$ be the solution of 
\begin{align*}
(P_s) \begin{cases}\nabla\cdot(\gamma_s\nabla u_s)=0 &\textnormal{ in } X\\
u_s=\phi(s)f_0 &\textnormal{ in } \partial X.
\end{cases}
\end{align*}
Differentiating with respect to $s$, $u_s'$ has to solve
\begin{align*}
\begin{cases}\nabla\cdot(\gamma_s\nabla u_s')+\nabla\cdot(\gamma_s'\nabla u_s)=0 &\textnormal{ in } X\\
u_s'=\phi'(s)f_0 &\textnormal{ in } \partial X.
\end{cases}
\end{align*}
We want to construct $\phi$ such that 
$\nabla u_s(\hat{x})\cdot\nabla u_s'(\hat{x})\geq0, \forall s\in[0,1]$. Let $v_s,w_s$ be the
solutions of 
\begin{align*}
\begin{cases}\nabla\cdot(\gamma_s\nabla v_s)=0 &\textnormal{ in } X\\
v_s=\phi'(s)f_0 &\textnormal{ in } \partial X
\end{cases}
\end{align*}
and
\begin{align*}
\begin{cases}\nabla\cdot(\gamma_s\nabla w_s)+\nabla\cdot(\gamma_s'\nabla u_s)=0 &\textnormal{ in } X\\
w_s=0 &\textnormal{ in } \partial X.
\end{cases}
\end{align*}
Then $u_s'=v_s+w_s$ and (it can be checked that $\phi(s)\neq0$)
\begin{align*}
\nabla u_s(\hat{x})\cdot\nabla u_s'(\hat{x})&=\nabla u_s(\hat{x})\cdot\nabla v_s(\hat{x})+
\nabla u_s(\hat{x})\cdot\nabla w_s(\hat{x})\\
&=\frac{\phi'(s)}{\phi(s)}|\nabla u_s(\hat{x})|^2 +  \nabla u_s(\hat{x})\cdot\nabla w_s(\hat{x}).
\end{align*}
Hence, to have $\nabla u_s(\hat{x})\cdot\nabla u_s'(\hat{x})\geq 0$ we essentially need $\phi$ to satisfy
a condition of the form
\begin{align*}
\frac{\phi'(s)}{\phi(s)}&=\max\Big\{0,-\frac{\nabla u_s(\hat{x})\cdot\nabla w_s(\hat{x})}{|\nabla u_s(\hat{x})|^2}\Big\}.
\end{align*}
This condition implies the following estimate on $\phi$
\begin{align*}
|\frac{\phi'(s)}{\phi(s)}|&\leq C| u_s|_{k+2,\alpha,X}
\leq \tilde{C} | \phi(s)|,
\end{align*}
so that $$|\phi'(s)|\leq\tilde{C} |\phi(s)|^2.$$
In general, we cannot obtain any better estimate.
Such an estimate guarantees the existence of $\phi$ for $s$ in an open subset
of $[0,1]$, but it does not guarantee global existence in $[0,1]$.
Indeed, the existence of $\phi$ for $s\in[0,1]$ is equivalent to saying that
for all $s\in[0,1]$, the solution $u_s$ of 
\begin{align*}
\begin{cases}\nabla\cdot(\gamma\nabla u_s)=0 &\textnormal{ in } X\\
u_s=f_0 &\textnormal{ in } \partial X
\end{cases}
\end{align*}
satisfies $\nabla u_s(\hat{x})\neq 0$. Yet, it is known that critical points of elliptic solutions do occur; see, e.g. \cite{B-UMEIT-12,BMN-ARMA-04,HaHoHoNa-JDG99,M-PAMS-93,RS-BSI-89}.

This shows that $|f_s|$ may blow up in finite time if $|f_s'|$ is large enough. We thus need to construct $f_s$
in such a way that $|f_s'|$ remains sufficiently small. The
construction of $F$ provided in Definition \ref{definitionF} is obtained by requiring an optimality
condition in that sense.

\begin{theorem} For $f \in C^{k+2,\alpha}(\partial X)$ and $s\in[0,1]$ let $u$ be the solution of
\begin{align*}
\begin{cases}\nabla\cdot(\gamma_s\nabla u)=0 &\textnormal{ in } X\\
u=f &\textnormal{ in } \partial X.
\end{cases}
\end{align*}
Let $v^g$ denote the solution of 
\begin{align*}
\begin{cases}\nabla\cdot(\gamma_s\nabla v^g)+\nabla\cdot(\gamma_s'\nabla u)=0 &\textnormal{ in } X\\
v^g=g &\textnormal{ in } \partial X.
\end{cases}
\end{align*}
The construction of $F(f,s)$ in Definition \ref{definitionF} is such that
\begin{align*}
F(f,s)=\arg\min\big\{|g|^2_{L^2(\partial X)}: g\in C^{k+2,\alpha}(\partial X)  \wedge \nabla u_s(\hat{x})\cdot\nabla v^g(\hat{x})\geq 0\big\}.
\end{align*}
\end{theorem}

\begin{proof}
Let
\begin{align*}
\mathcal{G} & = \big\{g\in C^{k+2,\alpha}(\partial X) : \nabla u_s(\hat{x})\cdot\nabla v^g(\hat{x})\geq 0\big\}\\
\hat{g} & =\arg\min\big\{|g|^2_{L^2(\partial X)}: g\in \mathcal{G}\big\}
\end{align*}
$\mbox{}$From Theorems \ref{firstpropF} and \ref{secondpropF}, $F(f,s)\in\mathcal{G}$, hence $\mathcal{G}\neq \emptyset$.
Also $\mathcal{G}$ is convex and closed in $C^{k+2,\alpha}(\partial X)$. The objective function $|g|_{L^2(\partial X)}$
is strictly convex and coercive in $C^{k+2,\alpha}(\partial X)$. The existence of $\hat{g}$ does not automatically follow
from this, because $C^{k+2,\alpha}(\partial X)$ is not reflexive, but if $\hat{g}$ exists, then  it is unique.

Theorem \ref{existslambda} and Definition \ref{definitionF} imply the existence of $\tilde{\mu}=\min(0,\mu)\leq0$ and
$\lambda \in L^p(X)\subset \big(C^{k+2,\alpha}(\partial X)\big)^*$ with
$\partial \lambda/\partial\nu \in C^{k+2,\alpha}(\partial X)$ satisfying
\begin{align*}
\begin{cases}\nabla\cdot(\gamma_s\nabla \lambda)=\nabla u_s(\hat{x}) \delta_{\hat{x}}&\textnormal{ in } X\\
\lambda=0&\textnormal{ in } \partial X,
\end{cases}
\end{align*}
\begin{align*}
F(f,s)= \tilde{\mu} \gamma_s \partial \lambda/\partial \nu \big|_{\partial X}
\end{align*}
and
\begin{align*}
\tilde{\mu}\nabla u_s(\hat{x})\cdot \nabla v^{F(f,s)}(\hat{x}) =0.
\end{align*}
These are the Karush-Kuhn-Tucker (KKT) conditions \cite{NocedalWright99} for the problem defining $\hat{g}$. The existence
of the KKT multipliers $\lambda,\tilde{\mu}$ with the above conditions imply that $\hat{g} = F(f,s)$,
and in particular imply the existence of $\hat{g}$. The fact that the KKT conditions in a convex
problem imply optimality is easy to check in general. We briefly present the calculations in this particular
case for concreteness.

If $F(f,s)=0$, $F(f,s)$ is clearly the element in $\mathcal{G}$
of minimal norm. Otherwise, $\mu<0$ and $F(f,s)$ is such that
\begin{align*}
\nabla u_s(\hat{x})\cdot \nabla v^{F(f,s)}(\hat{x})=0.
\end{align*}
We recall that for any $g\in\mathcal{G}$
\begin{align*}
\nabla u_s(\hat{x})\cdot \nabla v^g(\hat{x})\geq 0.
\end{align*}
For any $g\in\mathcal{G}$, multiplication of the equation of $v^g$ by $\lambda$ and integration
by parts, gives
\begin{align*}
\nabla u_s(\hat{x})\cdot\nabla v^g(\hat{x}) = \int_X\lambda \nabla\cdot(\gamma_s'\nabla u_s)
- \int_{\partial X} \gamma_s \frac{\partial\lambda}{\partial\nu} g.
\end{align*}
Subtracting this last expression for $g\in \mathcal{G}$ and $F(f,s)$, we obtain
\begin{align*}
- \int_{\partial X} \gamma_s \frac{\partial\lambda}{\partial\nu} g
+ \int_{\partial X} \gamma_s \frac{\partial\lambda}{\partial\nu} F(f,s)&=
\nabla u_s(\hat{x})\cdot\nabla v^g(\hat{x}) - \nabla u_s(\hat{x})\cdot\nabla v^{F(f,s)}(\hat{x})\geq0\\
\Rightarrow -\int_{\partial X} \gamma_s \frac{\partial\lambda}{\partial\nu} F(f,s) &\leq - \int_{\partial X} \gamma_s \frac{\partial\lambda}{\partial\nu} g.
\end{align*}
Since $\mu<0$ and $F(f,s)=\mu\gamma_s\partial\lambda/\partial\nu \neq 0$ the previous inequality implies
\begin{align*}
|\mu| \Big|\gamma_s \frac{\partial\lambda}{\partial\nu}\Big|^2_{L^2(\partial X)}&
\leq |\int_{\partial X} \gamma_s \frac{\partial\lambda}{\partial\nu} g|\\
&\leq \Big|\gamma_s \frac{\partial\lambda}{\partial\nu}\Big|_{L^2(\partial X)}\Big| g\Big|_{L^2(\partial X)}\\
\Rightarrow \Big|\mu\gamma_s \frac{\partial\lambda}{\partial\nu}\Big|_{L^2(\partial X)}&
\leq | g|_{L^2(\partial X)}\\
\Leftrightarrow |F(f,s)|_{L^2(\partial X)}& \leq | g|_{L^2(\partial X)}
\end{align*}
proving that $F(f,s)$ is the element in $\mathcal{G}$ of minimal $L^2(\partial X)$ norm.
\end{proof}
In summary, among all the possible choices of $F$ satisfying the non-decreasing
norm of the gradient at $\hat{x}$, our definition of $F(f,s)$ is the one of minimal $L^2(\partial X)$ norm at
each $s\in[0,1]$.

\section{Proofs and Intermediate Results}\label{Proofs}

\subsection{Proof of Theorem \ref{frame}} In this subsection, let $k\in\N$ fixed, $0<\alpha< 1$ fixed.

\begin{theorem}\label{mainu} Let $X$ be a $C^{k+2,\alpha}$ bounded domain in $\R^n$.
Let $f \in C^{k+2,\alpha}(\partial X)$ and $h\in C^{k,\alpha}(\overline{X})$.
Let $\gamma\in C^{k+1,\alpha}(\overline{X})$ be such that $\exists c,C$ constants for which
\begin{align*}
  0<c\leq\gamma(x)\leq C<\infty \hspace{1cm} \forall x\in X.
\end{align*}
Then there is a unique solution  $u\in C^{k+2,\alpha}(\overline{X})$ of the equation
\begin{align*}
\begin{cases}\nabla\cdot(\gamma\nabla u)=h &\textnormal{ in } X\\
u=f &\textnormal{ in } \partial X,
\end{cases}
\end{align*}
and $u$ satisfies the following estimate where the constant $\kappa$ depends only on $n, \alpha, c, C$ and $X$,
$$|u|_{k+2,\alpha,X}\leq \kappa\Big(|f|_{k+2,\alpha,\partial X}+|h|_{k,\alpha,X}\Big).$$
\end{theorem}

For the previous Theorem, existence is established in \cite[Thm. 6.14, Thm. 6.19, Lem. 6.38]{gt1}
and the estimate is a consequence of \cite[Thm. 6.6, Lem. 6.38, Thm. 3.7]{gt1}.
The estimate and the linearity of the problem imply a smooth dependence
of $u$ with respect to the boundary condition and the equation coefficient. This is stated
explicitly as follows.

\begin{corollary}\label{derivu} Let $X$ be a $C^{k+2,\alpha}$ bounded domain in $\R^n$.
For an interval $I\subset\R$ let $s\mapsto f_s \in C^1(I;C^{k+2,\alpha}(\partial X))$
and $s\mapsto h_s\in C^1(I;C^{k,\alpha}(\overline{X}))$.
Let $s\mapsto \gamma_s\in C^1(I;C^{k+1,\alpha}(\overline{X}))$ and
such that $\exists c,C$ constants for which
\begin{align*}
0<c\leq\gamma_s(x)\leq C<\infty \hspace{1cm} \forall x\in X, \forall s\in I.
\end{align*}
If we let $u_s,s\in I$, be the solutions of
\begin{align*}
\begin{cases}\nabla\cdot(\gamma_s\nabla u_s)=h_s &\textnormal{ in } X\\
u_s(x)=f_s(x) &\textnormal{ in } \partial X,
\end{cases}
\end{align*}
then $s\mapsto u_{s}\in C^1(I;C^{k+2}(\overline{X}))$. Letting $\gamma_s':=\partial \gamma_s/\partial s$ ,
$f_s':=\partial f_s/\partial s$ and $h_s':=\partial h_s/\partial s$,
we also get that $u'_s:=\partial u_s/\partial s$ satisfies the equation 
\begin{align*}
\begin{cases}
\nabla\cdot(\gamma_s\nabla u_s'(x))+\nabla\cdot(\gamma_s'\nabla u_s(x))=h_s' &\textnormal{ in } X\\
u_s'(x)=f_s'(x) &\textnormal{ in } \partial X.
\end{cases}
\end{align*}
In addition,  for a given $\hat{x}\in X$, we have $\frac{d}{ds}\Big(|\nabla u_s(\hat{x})|^2\Big)=2\nabla u_s(\hat{x})\cdot \nabla u_s'(\hat{x})$.
\end{corollary}

\begin{proof}[{\bf Proof of Theorem \ref{frame}}] It is a direct consequence of Theorems \ref{mainu} and
Corollary \ref{derivu}.
\end{proof}

\begin{remark}
\label{rem:galell}\rm Theorem \ref{mainu} and Corollary \ref{derivu} remain true if we replace $\nabla\cdot(\gamma\nabla)$ by
any uniformly elliptic operator $L=a_{ij}\partial_{x_i}\partial_{x_j}+b_i\partial_{x_i}$ with $a_{ij}, b_i\in C^{k,\alpha}(\overline{X})$.
\end{remark}

\subsection{Proof of Theorem \ref{existslambda}}  In this subsection let $k\in\N$ fixed. Let $p\in(1,\frac{n}{n-1})$ fixed.
Let $\alpha=(n-\frac{n}{p})\in(0,1)$.\\

For $y\in\R$ let $\partial_y=(y\cdot\nabla)$, $s\in[0,1]$, we study the auxiliary problem.
\begin{align*}
(A_s)\begin{cases}\nabla\cdot(\gamma_s\nabla \lambda)=\partial_y\delta_{\hat{x}}&\textnormal{ in } X\\
\lambda=0&\textnormal{ in } \partial X.
\end{cases}
\end{align*}
Intuitively, the solution $\lambda$ is a directional derivative of a Green's function, and so it should behave as a 
Green's function with one degree less of regularity.
Among the statements in Theorem \ref{existslambda}, the uniqueness of $\lambda$ is the simplest and follows from
standard arguments.
The continuous dependence of $\lambda$ on $s$ is the most technical aspect and it will require the explicit construction of
the singular part of $\lambda$. This construction will also prove the existence and regularity stated in Theorem 
\ref{existslambda}. The construction of the singular part of $\lambda$ is presented in a couple of technical lemmas below.
We start by introducing the necessary notation.

\begin{Def} Let $E\subset\N$. We say that $\{c_j\}_{j\in E}$ is a family of homogeneous
polynomials centered at $\hat{x}$ if each $c_j$ is a polynomial formed exclusively by monomials
centered at $\hat{x}$ of total degree $j$, namely
$$c_j(x)=\sum_{|\beta|=j} c_{\beta,j}(x-\hat{x})^\beta$$ where
$\beta\in\N^n$, $|\beta|=\sum_{i=1}^n\beta_i$,
$(x-\hat{x})^\beta=\Pi_{i=1}^n(x_i-\hat{x}_i)^{\beta_i}$ and each $c_{\beta,j}\in \R$.
We say that $\{c_{\beta,j}\}_{|\beta|=j}\subset \R$ are the (finitely many) coefficients of $c_j$.
\end{Def}

\begin{Def}\label{familyvm} Let $E=\N$ or $E=\{0,1,2,...,n\}$. Let $\{c_j\}_{j\in E}$ be a family of homogeneous
polynomials centered at $\hat{x}$ with $c_0\neq0$. We define the family of functions $\{v_m\}_{m\in E}$
associated to  $\{c_j\}_{j\in E}$  as follows.

Let $B$ be an open ball centered in $\hat{x}$ and containing $\overline{X}$.
Let $g$ be the solution of
\begin{align*}
\begin{cases}
\Delta g=\delta_{\hat{x}}&\textnormal{ in } B\\
g=0 &\textnormal{ in } \partial B.
\end{cases}
\end{align*}
Then define $v_0:=\frac{1}{c_0}\partial_y g$.\\

Let $w$ be the solution of
\begin{align*}
\begin{cases}
\Delta w=v_0&\textnormal{ in } B\\
w=0 &\textnormal{ in } \partial B.
\end{cases}
\end{align*}
Then define $v_1:=\frac{1}{c_0}(\nabla c_1\cdot\nabla w-c_1 v_0)$.\\

For $2\leq m, m\in E$, define recursively $v_m$ as the solution of 
\begin{align*}
\begin{cases}
c_0\Delta v_m=\sum_{i=0}^{m-1}[\nabla\cdot(v_i\nabla c_{m-i})-\Delta(c_{m-i}v_i)]&\textnormal{ in } B\\
v_m=0 &\textnormal{ in } \partial B.
\end{cases}
\end{align*}
\end{Def}

\begin{lemma} A family $\{v_m\}_{m\in E}$ from
Definition \ref{familyvm} satisfies
\begin{enumerate}
\item[(a)] $v_0\in L^p(B)\cap C^\infty(B\setminus\{\hat{x}\})$.\\
\item[(b)] $d_jv_0 \in W^{j,p}(B)$ for any $d_j$ homogeneous polynomial of
degree $j$ centered at $\hat{x}$, $\forall j\in\N$.\\
\item[(c)] $v_1\in W^{1,p}(B)\cap C^\infty(B\setminus\{\hat{x}\})$.\\
\item[(d)] We get  in $X$
\begin{align}
\label{eqv0} &\nabla\cdot( c_0\nabla v_0)=\partial_y\delta_{\hat{x}}\\
\label{cancelvm}&\sum_{i=0}^{m} \nabla\cdot( c_{m-i}\nabla v_i)=0,\forall m\geq 1
\end{align}
\end{enumerate}
\end{lemma}

\begin{proof} Using the same notation as in Definition \ref{familyvm},
$g$ is the Green's function at $\hat{x}$ of the Direchlet problem for
the Laplacian in $B$, hence there is an explicit expression of $g$;
see \cite{gt1}. Since $c_0$ is constant and
$v_0:=\frac{1}{c_0}\partial_y g$, properties (a) and (b)
are automatically verified from the explicit expression for $v_0$
(recall that $p\in(1,\frac n{n-1})$).

Property (a) and the definition of $w$ imply $w\in W^{2,p}(B)\cap
C^\infty(B\setminus\{\hat{x}\})$; see \cite[Thm 9.15]{gt1}. Then (c) follows from the definition of $u_1$ and (b).

Property (d) is the definition of $v_m$ rewritten.
\end{proof}

The family $\{v_m\}_{m\in E}$ has the following regularity. 

\begin{lemma}\label{singseries} Let $\{c_j\}_{j\in E},\{d_j\}_{j\in E}$ be families of homogeneous polynomials
centered at $\hat{x}$. Let $\{v_m\}_{m\in E}$ be the family of functions associated to $\{c_j\}_{j\in E}$ in
Definition \ref{familyvm}. Then
\begin{enumerate}
\item $v_m\in W^{m,p}(B)\cap C^\infty(B\setminus\{\hat{x}\}), \forall m\in E$.\\
\item $d_j v_m\in W^{j+m,p}(B), \forall m\in E, \forall j\geq1$.\\
\item $d_j\Delta v_m \in W^{m+j-2,p}(B), \forall j\geq1,m\geq1$.
\end{enumerate}
\end{lemma}

\begin{proof} The proof is by induction. As the base case, from the previous Lemma we already have (1) for $m=0$ and (2) for
$m=0,\forall j\geq1$. We also have (1) for $m=1$. The following steps complete the induction argument.\\

{$\Big[(1) \forall 0\leq m< M \textnormal{ and } (2) \forall j\geq1, \forall 0\leq m< M\Big]\Rightarrow
\Big[(3) \textnormal{ for }m=M,\forall j\geq1\Big].$}
Using the  definition of $v_{M}$, $M\geq1$
\begin{align*}
d_j\Delta v_{M}=&\frac{d_j}{c_0}\sum_{i=0}^{M-1}[\nabla\cdot(v_i\nabla c_{M-i})-\Delta(c_{M-i}v_i)]\\
=&\frac{1}{c_0}\sum_{i=0}^{M-1}[\nabla\cdot(d_jv_i\nabla c_{M-i})-v_i\nabla\cdot(d_j\nabla c_{M-i})-d_j
\Delta(c_{M-i}v_i)]\\
=&\frac{1}{c_0}\sum_{i=0}^{M-1}[\nabla\cdot(d_jv_i\nabla c_{M-i})-v_i\nabla\cdot(d_j\nabla c_{M-i})-
\Delta(d_jc_{M-i}v_i)\\
&-c_{M-i}v_i\Delta d_j +2\nabla\cdot(c_{M-i} v_i \nabla d_j)]
\end{align*}
and by the inductions hypotheses each summand in the right hand side is in $W^{j+M-2,p}(B)$.\\

{For $M\geq1$,  $\Big[(1) \textnormal{ and } (3) \textnormal{ for }  j=1\Big]\Rightarrow
\Big[(2) \textnormal{ for } j=1\Big].$} We have
\begin{align*}
\Delta d_1v_M&=2\nabla\cdot(v_M\nabla d_1)+d_1\Delta v_M
\end{align*}
and by induction hypotheses the right hand side is in $W^{M-1,p}(B)$, hence (Chp. 9, \cite{Krylov08})
$d_1v_M\in W^{M+1,p}(B)$.\\

{For $M\geq1$,  $\Big[(3) \textnormal{ for } j=J \textnormal{ and } (2) \textnormal{ for } 1\leq j<J \Big]\Rightarrow
\Big[(2) \textnormal{ for } J\Big].$} We have
\begin{align*}
\Delta d_Jv_M&=2\nabla\cdot(v_M\nabla d_J)+d_J\Delta v_M-v_M\Delta d_J
\end{align*}
and by induction hypotheses the right hand side is in $W^{M+J-2,p}(B)$, hence (Chp. 9, \cite{Krylov08})
$d_Jv_M\in W^{M+J,p}(B)$.\\

{$\Big[(1) \forall 1\leq m< M \textnormal{ and } (2) \forall j\geq1, \forall 1\leq m< M\Big]\Rightarrow
\Big[(1) \textnormal{ for }m=M\Big].$}
By definition 
\begin{align*}
\Delta v_{M}&=\frac{1}{c_0}\sum_{i=0}^{M-1}[\nabla\cdot(v_i\nabla c_{M-i})-\Delta(c_{M-i}v_i)].
\end{align*}
By induction hypotheses, for$ M\geq 2$ the right hand side is in $W^{M-2,p}(B)$, hence
$v_M\in W^{M,p}(B)$. Elliptic regularity and the induction hypotheses also imply
$v_M\in C^\infty(B\setminus\{\hat{x}\})$.
\end{proof}

In Definition \ref{familyvm} we have an explicit construction of each $v_m$
in terms of the polynomials $c_j$. This provides an explicit dependence of
each $v_m$ in terms of the coefficients of the $c_j$'s.

\begin{lemma}\label{basis} Let $\{v_m\}_{m\in E}$ be the family associated to $\{c_j\}_{j\in E}$.
We can write each $v_m$ as
\begin{align*}
v_m=\sum_{l\in I_m} p_{l,m} e_{l,m}
\end{align*}
where $I_m$ is a finite index set, $\{p_{l,m}\}_{l\in I_m}$ is a family of real valued polynomials evaluated
in $\{1/c_0\}\cup\{c_{\beta,j}\}_{|\beta|=j,1\leq j\leq m}$, but otherwise independent of $\{c_j\}$.
And where $\{e_{l,m}\}_{l\in I_m}$ is a family of functions in $X$ independent
of $\{c_j\}$, each $e_{l,m}$ satisfying (1),(2),(3) for $m$ of Lemma \ref{singseries}.
\end{lemma}

\begin{proof} By induction. True for $v_0$ from its definition with $p_{1,0}(1/c_0)=1/c_0$ and $e_{1,0}=\partial_y g$. For $m\geq1$,
the linear system defining $v_m$ can be written as
\begin{align*}
\Delta v_m &= - \frac{1}{c_0}\sum_{i=0}^{m-1} \nabla(  c_{m-i}\nabla v_i)\\
\Delta v_m &= \frac{1}{c_0}\sum_{i=0}^{m-1} \sum_{|\beta|={m-i}} c_{\beta,m-i} \nabla((x-\hat{x})^\beta \nabla v_i)\\
&= \sum_{i=0}^{m-1} \sum_{|\beta|={m-i}} \sum_{l\in I_i} \frac{c_{\beta,m-i}}{c_0}p_{l,i}   \nabla((x-\hat{x})^\beta \nabla e_{l,i}).
\end{align*}
Defining $\{e_{l,m}\}_{l\in I_m}$ as the solutions $e$ of the family of equations  
\begin{align*}
\Delta e &= \nabla((x-\hat{x})^\beta \nabla e_{l,i}) \hspace{.5cm} \textnormal{ for } 0\leq i\leq m-1,  |\beta|=i,  l\in I_i
\end{align*}
the result follows.
\end{proof}

We can now explicitly describe the singular part of the solution $\lambda$ of $(A_s)$.

\begin{theorem}\label{singpart} Let $\gamma\in C^K(\overline{B})$. Let $\{c_j\}_{j=0}^K$ form  the partial
Taylor sum of $\gamma$ about $\hat{x}$, namely
\begin{align*}
\gamma(x)=\sum_{j=0}^K c_j(x)+\gamma_K(x)
\end{align*}
with $\gamma_K(x)=o(|x-\hat{x}|^K), \gamma_K\in C^K(\overline{B})$. Assume $c_0\neq 0$ and let $\{v_m\}_{m=1}^K$ 
be the family constructed in Definition \ref{familyvm}, corresponding to the $\{c_j\}_{j=0}^K$.
Define $w_K\in L^p(B)\cap C^{\infty}(B\setminus\{\hat{x}\})$ as
\begin{align*}
w_K=\sum_{m=0}^K v_m.
\end{align*}
Then there exists $h_K\in W^{K-2,p}(B)\cap C^{\infty}(B\setminus\{\hat{x}\})$ such that
\begin{align*}
\nabla\cdot(\gamma\nabla w_K)=\partial_y \delta_{\hat{x}}+h_K&\textnormal{ in } B.
\end{align*}
In addition, if $U$ is a compact subset of $ B\setminus\{\hat{x}\}$, then $w_K\in L^p(B)\cap C^\infty(U)$
depends continuously in the coefficients of  $\{c_j\}_{j=0}^K$.
Also, $h_K\in W^{K-2,p}(B)\cap C^\infty(U)$ depends continuously in $\gamma$ under $C^K(\overline{B})$
perturbations.
\end{theorem}

\begin{proof} Let
\begin{align*}
\gamma_{K-i}=\gamma(x)-\sum_{j=0}^{K-i} c_j(x).
\end{align*}
Then $\gamma_{K-i}(x)=o(|x-\hat{x}|^{K-i}), \gamma_{K-i}\in C^K(\overline{B})$. We have
\begin{align*}
\nabla\cdot(\gamma\nabla w_K)&=\sum_{m=0}^K\nabla\cdot(\gamma \nabla v_m)\\
&=\sum_{m=0}^K\nabla\cdot\Big(\Big[\sum_{j=0}^{K-m} c_j+\gamma_{K-m}\Big] \nabla v_m\Big)\\
&=\sum_{m=0}^K\sum_{j=0}^{K-m}\nabla\cdot( c_j \nabla v_m)
+\sum_{m=0}^K\nabla\cdot(\gamma_{K-m}\nabla v_m)
\end{align*}
\begin{align*}
\nabla\cdot(\gamma\nabla w_K)&=\sum_{i=0}^K\sum_{j=0}^{i}\nabla\cdot( c_{i-j} \nabla v_j)
+\sum_{m=0}^K\nabla\cdot(\gamma_{K-m}\nabla v_m)\\
&=\partial_y\delta_{\hat{x}}+h_K
\end{align*}
where the first term is simplified using equations (\ref{eqv0}), (\ref{cancelvm}) and where $h_K$ is defined as
\begin{align*}
h_K:= \sum_{m=0}^K\nabla\cdot(\gamma_{K-m}\nabla v_m).
\end{align*}
Then  $h_K  \in W^{K-2,p}(B)\cap C^{\infty}(B\setminus\{\hat{x}\})$ by Lemma \ref{singseries}. The continuous
 dependencies of $w_K$ and $h_K$ are a consequence of
Lemma \ref{basis} and the definitions of $w_K,h_K$.
\end{proof}

\begin{theorem}\label{lambda}
Let $X$ be a $C^{k+2,\alpha}$ bounded domain in $\R^n$. Let $\gamma\in C^{k+n+2}(\overline{X})$
be such that $\exists c,C$ constants for which
\begin{align*}
0<c\leq\gamma(x)\leq C<\infty \hspace{1cm} \forall x\in X .
\end{align*}
Then there is a solution $\lambda\in L^p(X)\cap C^{k+2,\alpha}(\overline{X}\setminus\{\hat{x}\})$ of
\begin{align*}
(A)\begin{cases}\nabla\cdot(\gamma\nabla \lambda)=\partial_y\delta_{\hat{x}}&\textnormal{ in } X\\
\lambda=0&\textnormal{ in } \partial X.
\end{cases}
\end{align*}
Also, for any compact
set $U\subset (\overline{X}\setminus\{\hat{x}\})$, we have that $\lambda|_U\in L^p(X)\cap C^{k+2,\alpha}(U)$ depends
continuously in $\gamma$ under $C^{k+n+2}(\overline{X})$ perturbations.
\end{theorem}

\begin{proof} Let $B$ be a ball centered in $\hat{x}$ and large enough to contain $\overline{X}$.
Extend $\gamma$ as $C^{k+n+2}(\overline{B})$ and let $K=k+n+2$. Let $\{c_j\}_{j=0}^K$ form the partial
Taylor series of $\gamma$ about $\hat{x}$ and let $w_K, h_K$ be as in Theorem \ref{singpart}.
Since $h_K\in W^{K-2,p}(B)$ then (Chp. 9, \cite{Krylov08}) there exists a unique
$v\in W^{K,p}(B)$ solution of 
\begin{align*}
\begin{cases}
\nabla\cdot(\gamma\nabla v)=-h_K&\textnormal{ in } B\\
v(x)=0&\textnormal{ in } \partial B
\end{cases}
\end{align*}
which depends continuously on $h_K$. By Sobolev embedding, $v\in C^{k+2,\alpha}(B)$
(recall that $\alpha=n-n/p$)
and it depends continuously on $h_K$, hence it depends continuously on $\gamma$ under
$C^{k+n+2}(\overline{B})$ perturbations.

Additionally, since $[w_K+v]_{\partial X}\in C^{k+2,\alpha}(\partial X)$, Theorem \ref{mainu}
implies that there is a unique $w\in C^{k+2,\alpha}(X)$ solution of
\begin{align*}
\begin{cases}
\nabla\cdot(\gamma\nabla w)=0&\textnormal{ in } X\\
w(x)=-w_K(x)-v(x)&\textnormal{ in } \partial X
\end{cases}
\end{align*}
which depends continuously on $[w_K+v]_{\partial X}\in C^{k+2,\alpha}(\partial X)$,
hence it depends continuously in $\gamma$ under $C^{k+n+2}(B)$ perturbations.

Finally, $\lambda=(w_K|_{\overline{X}}+v|_{\overline{X}}+w)$ is a solution of $(A)$
with the desired properties.
\end{proof}

\begin{proof}[{\bf Proof of Theorem \ref{existslambda}}] Theorem \ref{existslambda} is Theorem \ref{lambda} for $k+1$ instead of $k$.
\end{proof}

\subsection{Proof of Theorem \ref{firstpropF}}   We use the same notation as Definition
\ref{definitionF} and Theorem \ref{firstpropF}.
\begin{proof}
We separate in two cases.
\begin{enumerate}
\item[{\bf Case 1}.] If $\nabla u(\hat{x})=0$ then immediately $\nabla u(\hat{x})\cdot\nabla v(\hat{x})=0$.\\
\item[{\bf Case 2}.] The equations for $u,v$ and $\lambda$, plus integration by parts, give
\begin{align*}
\nabla u(\hat{x})\cdot\nabla v(\hat{x})=\int_X\lambda \nabla\cdot((\gamma-\gamma_0)\nabla u)
-\mu\int_{\partial X}\gamma_s\frac{\partial \lambda}{\partial \nu}g
\end{align*}
with $g=F(f,s)$. Recall the definition of $\mu$,
\begin{align*}
\mu= \frac{\int_X\lambda\nabla \cdot((\gamma-\gamma_0)\nabla u)}{\qquad \Big|\gamma_s\frac{\partial\lambda}{\partial\nu}\Big|^2_{L^2(\partial X)}}.
\end{align*}
\subitem {\bf Case 2.1}. If $\mu\geq0$ then $g\equiv 0$ and $\int_X\lambda \nabla\cdot((\gamma-\gamma_0)\nabla u)\geq0$, hence $\nabla u(\hat{x})\cdot\nabla v(\hat{x})=\int_X\lambda \nabla\cdot((\gamma-\gamma_0)\nabla u)\geq0$.\\
\subitem {\bf Case 2.2}.  If $\mu\leq0$ then $g=\mu\partial\lambda/\partial\nu$ and we get
$\nabla u(\hat{x})\cdot \nabla v(\hat{x})=0$.
\end{enumerate}
\end{proof}

\subsection{Proof of Theorem  \ref{injectlambda}} We use the notation of Theorem  \ref{injectlambda}.
\begin{proof}
It is clear that $\big[y=0\big]\Rightarrow\big[\lambda\equiv0\big]\Rightarrow \big[\gamma_s\partial\lambda/\partial\nu\big|_{\partial X}\equiv 0\big]$. In the opposite direction. Assume  $\gamma_s\partial\lambda/\partial\nu\big|_{\partial X}=0$, then
$\lambda$ satisfies the equation
\begin{align*}
\begin{cases}\nabla\cdot(\gamma_s\nabla \lambda)= 0 &\textnormal{ in } X\setminus\{\hat{x}\}\\
\lambda=0&\textnormal{ in } \partial X\\
\gamma_s\frac{\partial\lambda}{\partial\nu}=0&\textnormal{ in } \partial X.
\end{cases}
\end{align*}
By unique continuation $\lambda\equiv 0$ in $X\setminus\{\hat{x}\}$. Since $\lambda\in L^p(x)$
we conclude $\lambda\equiv 0$. Finally, if $y\neq0$ let $\varphi\in C_0^\infty(X)$ be such that
$\partial_y\varphi (\hat{x})\neq0$. Then $\lambda\neq 0\in L^p(X)$ since $\int_X\lambda \nabla\cdot(\gamma_s\nabla \varphi)=\partial_y\varphi (\hat{x})\neq0$.
\end{proof}

\subsection{Proof of Corollary \ref{estlambda}} We start with a lemma about injective linear maps defined over
a finite dimensional domain.

\begin{lemma}\label{linfin}
Let $I\subset\R$ be a closed bounded interval. Let $(V,|\cdot|_V)$ be a normed vector space.
Let $H_s:\R^n\to V, s\in I,$ be a family of injective linear functionals. Assume
\begin{align*}
\lim_{t\in I,t\to s} H_t(y)=H_s(y), \forall y\in \R^n,\forall s\in I.
\end{align*}
Then there exist constants $0<a,b<\infty$ such that $\forall s\in I$ 
\begin{align*}
a|y|\leq|H_s y|_V\leq b|y|,  \quad \forall y\in \R^n.
\end{align*}
\end{lemma}
 
\begin{proof} Let $\{e_i\}_{i=1}^n$ be a basis of $\R^n$. Since $I\ni s\mapsto H_s(e_i)$ are continuous and $I$ is compact,
$\max_{i=1,...,n}\sup_{s\in I} |H_s e_i|<\infty$. Since $H_s$ are linear, the existence of $b>0$ for the second inequality
follows.

Assume $\nexists a>0$ such that the first inequality holds. By the compactness of $I$ and
the linearity of each $H_t$, there exists $s\in I$ and $I\ni t\to s$, together with
$y_t\in \R^n,|y_t|=1, y_t \stackrel{t\to s}{\longrightarrow} y_s$,  such that $|H_t y_t|\stackrel{t\to s}{\longrightarrow} 0$.
 But then (using $|H_t (y_s-y_t)|_V\leq b|y_s-y_t|$)
\begin{align*}
0\leq |H_s y_s|_V\leq |H_s(y_s) - H_t(y_s)|_V+|H_t (y_s-y_t)|_V+ |H_t y_t|_V\stackrel{t\to s}{\longrightarrow}0.
\end{align*}
Hence $ H_s y_s=0$, contradicting the injectivity of $H_s$ since $|y_s|=1$.
\end{proof}

\begin{proof}[{\bf Proof of Corollary \ref{estlambda}}] From Theorem \ref{injectlambda},
the linear maps $\R^n\ni y\mapsto \lambda \in L^p(X)$ and
$\R^n\ni y\mapsto \gamma_s\partial\lambda/\partial\nu \in C^{k+2,\alpha}(\partial X)\subset L^2(\partial X)$
are injective $\forall s\in[0,1]$ ($\lambda$ is the solution of $(A_s)$). From Theorem \ref{existslambda},
for $y\in \R^n$ fixed, $\lambda \in L^p(X)$ and $\gamma_s\partial\lambda/\partial\nu \in
C^{k+2,\alpha}(\partial X)\subset L^2(\partial X)$ depend continuously on $s\in[0,1]$. Lemma \ref{linfin}
then implies that all the quantities
\begin{align*} |y|, |\lambda|_{L^p(X)}, |\gamma_s\partial \lambda/\partial\nu|_{k+2,\alpha,\partial X}
\textnormal{ and } |\gamma_s\partial \lambda/\partial\nu|_{L^2(\partial X)}
\end{align*}
are comparable uniformly $\forall s\in[0,1]$. The last statement of Corollary \ref{estlambda}
is true for any  quotient of two Lipschitz function in a set where the denominator is bounded
away from zero.
\end{proof}

\subsection{Proof of Theorem \ref{secondpropF}} Let us recall the definition of $F:C^{k+2,\alpha}(\partial X)\times[0,1]\to C^{k+2,\alpha}(\partial X)$. Given $f\in C^{k+2,\alpha}(\partial X)$ and
$s\in[0,1]$, let $u\in C^{k+2,\alpha}(X)$ be the solution of
\begin{align*}
\begin{cases}\nabla\cdot(\gamma_s\nabla u)=0 &\textnormal{ in } X\\
u=f &\textnormal{ in } \partial X.
\end{cases}
\end{align*}
Let $\lambda$ be the solution of
\begin{align*}
\begin{cases}\nabla\cdot(\gamma_s\nabla \lambda)=\nabla u(\hat{x}) \cdot \nabla \delta_{\hat{x}}&\textnormal{ in } X\\
\lambda=0&\textnormal{ in } \partial X.
\end{cases}
\end{align*}
If $\nabla u(\hat{x}) = 0$ let $\mu>0$, otherwise let
\begin{align*}
\mu= \frac{\int_X\lambda\nabla \cdot((\gamma-\gamma_0)\nabla u)}{\qquad \Big|\gamma_s\frac{\partial\lambda}{\partial\nu}\Big|^2_{L^2(\partial X)}}.
\end{align*}
We defined
\begin{align*}
F(f,s):= \begin{cases}
0& \textnormal{ if } \mu\geq 0\\
\mu \gamma_s(\frac{\partial\lambda}{\partial \nu})& \textnormal{ if } \mu\leq 0.
\end{cases}
\end{align*}

\begin{lemma}\label{boundF}
 There exists a constant $\kappa>0$ independent of $s\in[0,1]$ such that
\begin{align*}
|F(f,s)|_{k+2,\alpha,\partial X}\leq \kappa |f|_{k+2,\alpha,\partial X}.
\end{align*}
\end{lemma}
\begin{proof}

\end{proof} When $F(f,s)\equiv0$ there is nothing to prove. Otherwise ($\lambda\neq0$)
\begin{align*}
|F(f,s)|_{k+2,\alpha,\partial X}&=\Big|\mu \gamma_s\frac{\partial\lambda}{\partial \nu}\Big|_{k+2,\alpha,\partial X}\\
&=\frac{ \Big|\int_X\lambda\nabla \cdot((\gamma-\gamma_0)\nabla u)\Big|}{\qquad \Big|\gamma_s
\frac{\partial\lambda}{\partial\nu}\Big|^2_{L^2(\partial X)}}
\Big|\gamma_s\frac{\partial\lambda}{\partial \nu}\Big|_{k+2,\alpha,\partial X}\\
&\leq \Big|\int_X\frac{\lambda}{|\gamma_s
\frac{\partial\lambda}{\partial\nu}|_{L^2(\partial X)}}\nabla \cdot((\gamma-\gamma_0)\nabla u)\Big|
\frac{|\gamma_s\frac{\partial \lambda}{\partial \nu}|_{k+2,\alpha,\partial X}}
{|\gamma_s\frac{\partial\lambda}{\partial \nu}|_{L^2(\partial X)}}\\
&\leq \frac{|\lambda|_{L^p(X)}}{|\gamma_s\frac{\partial\lambda}{\partial \nu}|_{L^2(\partial X)}}
\frac{|\gamma_s\frac{\partial \lambda}{\partial \nu}|_{k+2,\alpha,\partial X}}
{|\gamma_s\frac{\partial\lambda}{\partial \nu}|_{L^2(\partial X)}}\Big|\nabla \cdot((\gamma-\gamma_0)\nabla u)
\Big|_{L^{p/(p-1)}(X)}\\
&\leq \tilde{\kappa} |u|_{2,X}\\
&\leq \kappa |f|_{k+2,\alpha,\partial X}.
\end{align*}
We used H\"older inequality to go from the third to the fourth line. Corollary \ref{estlambda} and the
boundedness of $X$ to go from the fourth
to the fifth line, and Theorem \ref{mainu} to go from the fifth to the last line.

\begin{Def} Given  $\eta>0$ and $s\in[0,1]$ let us define the set 
$N_{\eta,s}\subset C^{k+2,\alpha}(\partial X)$ as follows,
$f\in N_{\eta,s}$ if and only if the solution $u$
of the equation 
\begin{align*}
\begin{cases}\nabla\cdot(\gamma_s\nabla u)=0 &\textnormal{ in } X\\
u=f &\textnormal{ in } \partial X
\end{cases}
\end{align*}
satisfies $|\nabla u(\hat{x})|>\eta$.
\end{Def}

\begin{lemma}\label{lipschitz} Fix $\eta>0$, for $f\in N_{\eta,s}\subset C^{k+2,\alpha}(\partial X)$ and $s\in[0,1]$ let
$u, \lambda$ and $\mu$ be the ones involved in the definition of $F(f,s)$. Then
\begin{enumerate}
\item[$\bullet$] $N_{\eta,s}\ni f\mapsto \lambda/|\gamma_s\frac{\partial\lambda}{\partial \nu}|_{L^2(\partial X)}
\in L^p(X)$ is Lipschitz continuous and bounded, uniformly in $s\in[0,1]$.\\
\item[$\bullet$] $N_{\eta,s}\ni f\mapsto \gamma_s\frac{\partial\lambda}{\partial \nu}/|\gamma_s\frac{\partial\lambda}{\partial \nu}
|_{L^2(\partial X)} \in C^{k+2,\alpha}(\partial X)$ is Lipschitz continuous and bounded, uniformly in $s\in[0,1]$.\\
\item[$\bullet$] $N_{\eta,s}\ni f\mapsto u \in C^{k+2,\alpha}(X)$ is linear continuous, uniformly in $s\in[0,1]$.
\end{enumerate} 
\end{lemma}

\begin{proof}
The last property is a direct consequence of Theorem \ref{mainu}. The first two properties are quickly deduced from
Theorem \ref{mainu}, the definition of $\lambda$ and Corollary \ref{estlambda}.
\end{proof}

\begin{theorem}\label{lipschitzF}
Given $\eta>0$ there exists $\kappa>0$ such that $\forall s\in[0,1], \forall f_1,f_2\in N_{\eta,s}$
\begin{align*}
|F(f_1,s) - F(f_2,s)|_{k+2,\alpha,\partial X}\leq \kappa(1+|f_1|_{k+2,\alpha,\partial X}+|f_2|_{k+2,\alpha,\partial X})
|f_1 - f_2|_{k+2,\alpha,\partial X}.
\end{align*}
\end{theorem}

\begin{proof} Let $u_i,\lambda_i,\mu_i, i=1,2$ be the values appearing in the definitions of
$F(f_1,s)$ and $F(f_2,s)$ correspondingly.

If $\mu_1,\mu_2\geq 0$ then $F(f_1,s)\equiv F(f_2,s)\equiv 0$ and $|F(f_1,s) - F(f_2,s)|_{k+2,\alpha,\partial X}=0$.

If $\mu_1\geq 0$ and $ \mu_2\leq 0$ then $F(f_1,s)=0$ and
\begin{align*}
|F(f_1,s) - F(f_2,s)|_{k+2,\alpha,\partial X}&=|F(f_2,s)|_{k+2,\alpha,\partial X}\\
& = \Big|\int_X\frac{\lambda_2}{|\gamma_s\frac{\partial\lambda_2}{\partial\nu}|_{L^2(\partial X)}}\nabla \cdot((\gamma-\gamma_0)\nabla u_2)\Big|
\frac{|\gamma_s\frac{\partial \lambda_2}{\partial \nu}|_{k+2,\alpha,\partial X}}
{|\gamma_s\frac{\partial\lambda_2}{\partial \nu}|_{L^2(\partial X)}}\\
& \leq \rho \Big|\int_X\frac{\lambda_2}{|\gamma_s\frac{\partial\lambda_2}{\partial\nu}|_{L^2(\partial X)}}\nabla \cdot((\gamma-\gamma_0)\nabla u_2)\Big|\\
& \leq \rho \Big|\int_X\frac{\lambda_2}{|\gamma_s\frac{\partial\lambda_2}{\partial\nu}|_{L^2(\partial X)}}\nabla \cdot((\gamma-\gamma_0)\nabla u_2)\\
&\qquad \qquad - \int_X\frac{\lambda_1}{|\gamma_s\frac{\partial\lambda_1}{\partial\nu}|_{L^2(\partial X)}}\nabla \cdot((\gamma-\gamma_0)\nabla u_1)\Big|.
\end{align*}
$\mbox{}$From the second to the third line we used Corolarry \ref{estlambda}. From the third to the last line we used the fact
that each integral has the same sign as the corresponding $\mu_i$, and we are in the case of $\mu_i$'s
with opposite signs.

If $\mu_i,\mu_2\leq0$ then 
\begin{align*}
|F(f_1,s) - F(f_2,s)|_{k+2,\alpha,\partial X} = \Big|\int_X\frac{\lambda_1}{|\gamma_s\frac{\partial\lambda_1}{\partial\nu}|_{L^2(\partial X)}}\nabla \cdot((\gamma-\gamma_0)\nabla u_1)
\frac{\gamma_s\frac{\partial \lambda_1}{\partial \nu}}
{|\gamma_s\frac{\partial\lambda_1}{\partial \nu}|_{L^2(\partial X)}}\\
 \qquad - \int_X\frac{\lambda_2}{|\gamma_s\frac{\partial\lambda_2}{\partial\nu}|_{L^2(\partial X)}}\nabla \cdot((\gamma-\gamma_0)\nabla u_2)
\frac{\gamma_s\frac{\partial \lambda_2}{\partial \nu}}
{|\gamma_s\frac{\partial\lambda_2}{\partial \nu}|_{L^2(\partial X)}}\Big|_{k+2,\alpha,\partial X}.
\end{align*}
Using Lemma \ref{lipschitz} we observe that in any of the three cases, we are left with products of
bounded Lipschitz functions and one continuous linear function, all bounds being uniform in $s\in[0,1]$, which readily implies the estimate above.
\end{proof}

\begin{proof}[{\bf Proof of Theorem \ref{secondpropF}}] Let $\eta>0$ be such that $f_0\in N_{\eta,0}$
and let $\rho >0$ be such that $|F(f,s)|_{k+2,\partial X}\leq \rho |f|_{k+2,\partial X}$ for all
$f\in C^{k+2}(\partial X), \forall s\in[0,1]$ (such $\rho$ exists by Lemma \ref{boundF}).

In order to prove Theorem \ref{secondpropF}, it is enough to
show that there is $\{f_s\}_{s\in[0,1]}\in C\big([0,1];C^{k+2,\alpha}(\partial X)\big)$, such that $\forall s \in[0,1]$,
\begin{align}\label{inteq}
f_s = f_0 + \int_0^s F(f_\tau,\tau) d\tau.
\end{align}
Writing the initial value problem in this integral form, the uniqueness of the solution will
be consequence of the existence proof (Step 2 below, which uses a Banach fixed point argument),
the continuous differentiability in $s$ will be automatic from the continuity of $s\mapsto f_s$ and the
continuity of $F$ (Theorem \ref{lipschitzF}).

To prove that there exists $\{f_s\}_{s\in[0,1]}$ satisfying Equation (\ref{inteq}) $\forall s\in[0,1]$,
we follow the proof of Picard-Lindel\"of Theorem for ODEs with some small modifications.
The proof is done in two steps.
\begin{lemma}[Step 1] Let $0<t\leq1$. If $\{f_s\}_{s\in[0,t)} \subset C^{k+2,\alpha}(\partial X)$
satisfies Equation (\ref{inteq}) $\forall s\in[0,t)$, then $f_t:=(\lim_{s\to t^-}f_s) \in C^{k+2,\alpha}(\partial X)$
exists and $f_t\in N_{\eta,t}$ (starting with $f_0\in N_{\eta,0}$). 
\end{lemma}
\begin{proof}[Proof of Step 1] If $\forall s\in[0,t)$
\begin{align*}
f_s = f_0 + \int_0^s F(f_\tau,\tau) d\tau,
\end{align*}
then $\forall s\in[0,t)$
\begin{align*}
|f_s|_{k+2,\alpha,\partial X} &\leq |f_0|_{k+2,\alpha,\partial X} + \int_0^s |F(f_\tau,\tau)|_{k+2,\alpha,\partial X} d\tau\\
&\leq |f_0|_{k+2,\alpha,\partial X} + \rho \int_0^s  |f_\tau|_{k+2,\alpha,\partial X} d\tau
\end{align*}
hence $|f_s|_{k+2,\alpha,\partial X} \leq e^{\rho s}|f_0|_{k+2,\alpha,\partial X}, \forall s\in[0,t)$. In particular,
for $0\leq s_1\leq s_2< t$
\begin{align*}
|f_{s_2}-f_{s_1}|_{k+2,\alpha,\partial X} &\leq \int_{s_1}^{s_2} |F(f_\tau,\tau)|_{k+2,\alpha,\partial X} d\tau\\
& \leq \rho e^{\rho t}|f_0|_{k+2,\alpha,\partial X} |s_2-s_1|,
\end{align*}
i.e., $\{f_s\}_{s\in[0,t)}$ is a Cauchy limit as $s\to t^-$. Since $C^{k+2,\alpha}(\partial X)$ is complete,
$f_t:=(\lim_{s\to t^-}f_s) \in C^{k+2,\alpha}(\partial X)$ exists. The inequality for 
$|f_{s_2}-f_{s_1}|_{k+2,\alpha,\partial X}$ also implies 
$\{f_s\}_{s\in [0,t]}$ continuous on $s\in[0,t]$, hence continuously differentiable in $s$, and since $f_0\in N_{\eta,0}$,
Theorem \ref{firstpropF} implies $f_s\in N_{\eta,s},\forall s\in [0,t]$.
\end{proof}
\begin{lemma}[Step 2] If $f\in N_{\eta,t}$ then there exists $\epsilon>0$ and 
a unique $\{f_s\}_{s\in[t,t+\epsilon)}\in C\big([t,t+\epsilon); C^{k+2,\alpha}(\partial X)\big)$
such that
\begin{align*}
f_s = f + \int_t^s F(f_\tau,\tau) d\tau,\quad \forall s\in[t,t+\epsilon).
\end{align*}
\end{lemma}
\begin{proof}[Proof of Step 2] Let us recall that $f\in N_{\eta,t} \subset  C^{k+2,\alpha}(\partial X)$ if and only if
the solution $u$ of the equation 
\begin{align*}
\begin{cases}\nabla\cdot(\gamma_t\nabla u)=0 &\textnormal{ in } X\\
u=f &\textnormal{ in } \partial X
\end{cases}
\end{align*}
satisfies $|\nabla u(\hat{x})|>\eta$. The smooth dependence of $u$ in terms of
the boundary condition and the equation coefficient (Theorem \ref{mainu}), implies the
existence of $\epsilon>0$ and $\delta>0$ such that if $h\in C^{k+2,\alpha}(\partial X)$
satisfies $|h|_{k+\alpha,\partial X}<\delta$, then $(f+h)\in N_{\eta,s},\forall s\in[t,t+\epsilon)$.

Let us consider the following non-empty closed set of $C\big([t,t+\epsilon); C^{k+2,\alpha}(\partial X)\big)$
\begin{align*}
\mathcal{F}=\{ \{f_s\}_{s\in[t,t+\epsilon)} \in C\big([t,t+\epsilon); C^{k+2,\alpha}(\partial X)\big) :
f_t=f, \sup_{s\in[t,t+\epsilon)} |f_s-f|_{k+2,\alpha,\partial X}<\delta\}.
\end{align*}
If $\epsilon>0$ is small enough, we can define the following operator $T:\mathcal{F}\to \mathcal{F}$
\begin{align*}
T(\{f_s\})_\sigma:=T(\{f_s\}_{s\in[t,t+\epsilon)})_\sigma := f + \int_t^\sigma F(f_\tau,\tau)d\tau, \quad 
\forall \sigma\in[t,t+\epsilon).
\end{align*}
Let us verify that $\{T(\{f_s\})_\sigma\}\in \mathcal{F}$ for
$\{f_s\}\in \mathcal{F}$. First $T(\{f_s\})_t=f$,
also
\begin{align*}
|T(\{f_s\})_\sigma - f |_{k+2,\alpha,\partial X}&\leq \int_t^\sigma |F(f_\tau,\tau)|_{k+2,\alpha,\partial X}d\tau\\
&\leq \epsilon \rho \sup_{\tau\in[t,\sigma]}|f_\tau|_{k+2,\alpha,\partial X}\\
&< \epsilon \rho (|f|_{k+2,\alpha,\partial X}+\delta)\\
&<\delta \quad \textnormal{ (if $0<\epsilon$ small enough) },
\end{align*}
and for $t\leq \sigma\leq r<t+\epsilon$
\begin{align*}
|T(\{f_s\})_\sigma - T(\{f_s\})_r|_{k+2,\alpha,\partial X}&\leq \int_r^\sigma |F(f_\tau,\tau)|_{k+2,\alpha,\partial X}d\tau\\
&< |\sigma - r| \rho (|f|_{k+2,\alpha,\partial X}+\delta)\\
\end{align*}
and hence  $\sigma\mapsto T(\{f_s\})_\sigma \in C^{k+2,\alpha}(\partial X)$ is continuous
and $\{T(\{f_s\})_\sigma\}\in \mathcal{F}$.

In addition, if $\{f_s\},\{g_s\}\in \mathcal{F}$, then for any $\sigma\in[t,t+\epsilon)$
\begin{align*}
|T(\{f_s\})_\sigma - T(\{g_s\})_\sigma|_{k+2,\alpha,\partial X}&\leq \int_t^\sigma |F(f_\tau,\tau)-F(g_\tau,\tau)|_{k+2,\alpha,\partial X}d\tau\\
&\leq \epsilon L \sup_{s\in[t,t+\epsilon)}[(1+ |f_s|_{k+2,\alpha,\partial X}+ |g_s|_{k+2,\alpha,\partial X})
|f_s-g_s|_{k+2,\alpha,\partial X}]\\
&\leq \epsilon \kappa (1+ 2|f|_{k+2,\alpha,\partial X}+ 2\delta)) \sup_{s\in[t,t+\epsilon)}|f_s-g_s|_{k+2,\alpha,\partial X}\\
&\leq \frac{1}{2} \sup_{s\in[t,t+\epsilon)}|f_s-g_s|_{k+2,\alpha,\partial X}\quad \textnormal{ (if $0<\epsilon$ small enough)}.
\end{align*}
To go from the second to the third line we used Theorem \ref{lipschitzF} and the fact that $f_s,g_s\in N_{\eta,s}$
for all $s\in[t,t+\epsilon)$.

Hence, we have that $T:\mathcal{F}\to\mathcal{F}$ is a contraction in a closed subset of a Banach space. The Banach fixed point Theorem implies the existence of a unique fixed point, hence the proof of Step 2 is complete.
\end{proof}
Putting together Step 1 and Step 2, since $[0,1]$ is connected, we conclude the
existence of a unique $\{f_s\}_{s\in[0,1]}$ that solves Equation (\ref{inteq}) for all $s\in[0,1]$ (starting
with $f_0\in N_{\eta,0}$ for some $\eta>0$).
This family is continuous in $s$, therefore continuously differentiable in $s$. Completing the proof of Theorem \ref{secondpropF}.
\end{proof}

\section{Extensions}\label{Extensions}

The evolution scheme presented in the previous sections solves constructively the following problem:
given a smooth enough bounded domain $X$ and coefficient $\gamma$, and given any point $\hat{x}\in X$,
find a boundary condition $\hat{f}$ such that the solution $u$ of 
\begin{align*}
\begin{cases}\nabla\cdot(\gamma\nabla u)=0 &\textnormal{ in } X\\
u=\hat{f} &\textnormal{ in } \partial X
\end{cases}
\end{align*}
satisfies $|\nabla u(\hat{x})|\geq 1$. We now consider two possible extensions, one that imposes a condition
over finitely many  points instead of only one, and one that imposes a condition involving finitely many equations.

\subsection{Finitely Many Points.}

Given a bounded domain $X$, a coefficient $\gamma$ and finitely many different points $\{\hat{x}_i\}_{i\in I}\subset X$, the goal is
to find a boundary condition $\hat{f}$, such that the solution $u$ of the equation
\begin{align*}
\begin{cases}\nabla\cdot(\gamma\nabla u)=0 &\textnormal{ in } X\\
u=\hat{f} &\textnormal{ in } \partial X
\end{cases}
\end{align*}
satisfies $|\nabla u(\hat{x}_i)|\geq1, \forall i\in I$. The process is analogous to the case of one point. We
are  now considering multiple constraints to be satisfied although we have only one equation and one boundary condition to control.
The scheme proposes to start with an appropriate $\gamma_0, f_0$  (e.g. $\gamma_0\equiv1$ and $f_0(x_1,x_2,...,x_n)=x_1$)  and construct $f_s$ such that the solution $u_s$ of
\begin{align*}
\begin{cases}\nabla\cdot(\gamma_s\nabla u_s)=0 &\textnormal{ in } X\\
u_s=f_s &\textnormal{ in } \partial X
\end{cases}
\end{align*}
and the solution $u_s'$ of
\begin{align*}
\begin{cases}\nabla\cdot(\gamma_s\nabla u_s')+\nabla\cdot(\gamma_s'\nabla u_s)=0 &\textnormal{ in } X\\
u_s'=f_s' &\textnormal{ in } \partial X
\end{cases}
\end{align*}
satisfy $\frac{d}{ds}|\nabla u_s(\hat{x}_i)|^2=2\nabla u_s(\hat{x}_i)\cdot\nabla u_s'(\hat{x}_i)\geq0, \forall i\in I$. Again, we construct $f_s$
as the solution of an initial value problem
\begin{align*}
\begin{cases}
\frac{\partial f_s}{\partial s}=F(f_s,s)\\
f_s|_{s=0}=f_0
\end{cases}
\end{align*}
for an appropriate $F$ defined below.

We construct the functional $F:C^{k+2,\alpha}(\partial X)\times[0,1]\to C^{k+2,\alpha}(\partial X)$
as follows. We assume $X$ is a $C^{k+3,\alpha}$ bounded domain. Let $\gamma_0,\gamma\in C^{k+n+3}(\overline{X})$
and let $\gamma_s=[(1-s)\gamma_0+s\gamma],s\in[0,1]$. Assume $0<c<\gamma_0,\gamma<C<\infty$.
For $s\in[0,1]$ and for $f\in C^{k+2,\alpha}(\partial X)$
let $u$ be the $C^{k+2,\alpha}(X)$ solution of
\begin{align*}
\begin{cases}\nabla\cdot(\gamma_s\nabla u)=0 &\textnormal{ in } X\\
u=f &\textnormal{ in } \partial X.
\end{cases}
\end{align*}
Let $v$ be the $C^{k+2,\alpha}(X)$ solution of 
\begin{align*}
\begin{cases}\nabla\cdot(\gamma_s\nabla v)+\nabla\cdot((\gamma-\gamma_0)\nabla u)=0 &\textnormal{ in } X\\
v= g &\textnormal{ in } \partial X
\end{cases}
\end{align*}
where $g\in C^{k+2,\alpha}(\partial X)$ will be prescribed below.

The difference with the previous process appears in that we need to consider many auxiliary problems.
Let $\lambda_i, i\in I$, be the solutions of 
\begin{align*}
\begin{cases}\nabla\cdot(\gamma_s\nabla \lambda_i)=\nabla u(\hat{x}_i)\cdot\nabla\delta_{\hat{x}_i}&\textnormal{ in } X\\
\lambda_i=0&\textnormal{ in } \partial X.
\end{cases}
\end{align*}
$\mbox{}$From Theorem \ref{existslambda}, $\lambda_i\in L^p(X)$ and 
$(\gamma_s \partial\lambda_i/\partial\nu)\in C^{k+2,\alpha}(\partial X)$ depend continuously on $s$. Since the 
$\{\hat{x}_i\}_{i\in I}$ are different, the $\{\lambda_i\}_{i\in I}$
are linearly independent and the $\{\gamma_s\partial\lambda_i/\partial\nu\}_{i\in I}$ are linearly independent (as long as
$\nabla u(\hat{x}_i)\neq0,\forall i\in I$). This is proved exactly as in Theorem \ref{injectlambda}.

By integration by parts, we obtain for all $i\in I$ that
\begin{align*}
\nabla u(\hat{x}_i)\cdot\nabla v(\hat{x}_i)&=\int_X\lambda_i \nabla\cdot((\gamma-\gamma_0)\nabla u)
-\int_{\partial X}\gamma_s\frac{\partial \lambda_i}{\partial \nu}g,
\end{align*}
and we let $g \in \Span(\{\gamma_s\partial\lambda_i/\partial\nu\}_{i\in I})\subset C^{k+2,\alpha}(\partial X)$
be such that, $\forall i\in I$
\begin{align*}
\int_{\partial X}\gamma_s\frac{\partial \lambda_i}{\partial \nu}g
=\int_X\lambda_i \nabla\cdot((\gamma-\gamma_0)\nabla u).
\end{align*}
Hence, $\nabla u(\hat{x}_i)\cdot \nabla v(\hat{x}_i)=0,\forall i\in I$. Since the
$\{\gamma_s\partial\lambda_i/\partial\nu\}_{i\in I}$ are linearly independent in $L^2(\partial X)$, such a $g$
exists and is unique
(as long as $\nabla u(\hat{x}_i)\neq0,\forall i\in I$).

Also, by an extension of the finite dimensional argument leading to Corollary \ref{estlambda}
(adding the linear independence of the $\{\gamma_s\partial\lambda_i/\partial\nu\}_{i\in I}$),
there exist constants $\tilde{C_1},\tilde{C_2}$, independent of $s$ and $f$,
such that
\begin{align*}
\Big|g\Big|_{k+2,\alpha,\partial X}\leq \tilde{C_1}\Big|u\Big|_{2, X} \leq \tilde{C_2}\Big|f\Big|_{k+2,\alpha,\partial X} .
\end{align*}
By defining $F(f,s):=g$, the continuity and boundedness of
$F:C^{k+2,\alpha}(\partial X)\times[0,1]\to C^{k+2,\alpha}(\partial X)$ can be proven exactly as it was
done in the previous case
(observing that the evolution with $F$ keeps $|\nabla u_s(\hat{x}_i)|\geq1,\forall i\in I$). This provides the following result:
\begin{theorem}\label{solofODEpoints} Assume $X$ is a $C^{k+3,\alpha}$ bounded domain.
Let $\gamma_0,\gamma\in C^{k+n+3}
(\overline{X})$ and let $\gamma_s=[(1-s)\gamma_0+s\gamma],s\in[0,1]$. Assume $0<c<\gamma_0,\gamma<C<\infty$.
Let $f_0,\gamma_0$ be chosen appropriately (e.g. $\gamma_0\equiv1$ and $f_0(x_1,x_2,...,x_n)=x_1$) .
Define $F:C^{k+2,\alpha}(\partial X)\times[0,1]\to C^{k+2,\alpha}(\partial X)$  as above.
Then there exists a unique solution $\{f_s\}_{s\in[0,1]}$ in $C^1([0,1];C^{k+2,\alpha}(\partial X))$
of the initial value problem
\begin{align*}
\begin{cases}
\frac{\partial}{\partial s}f_s=F(f_s,s)\\
f_s|_{s=0}=f_0.
\end{cases}
\end{align*}
The family $\{f_s\}_{s\in[0,1]}$ constructed in this way, satisfies that each $u_s$, solution of $(P_s)$ with
boundary condition $f_s$, is such that $|\nabla u_s(\hat{x}_i)|\geq1,\forall i\in I,\forall s\in [0,1]$.
\end{theorem}

Hence $\hat{f}=f_s|_{s=1}$ solves the problem presented at the beginning of this Subsection, with a condition
imposed over finitely many points.

\subsection{System of Equations.}

Given a bounded domain $X\subset\R^3$, a coefficient $\gamma$ and fixed point $\hat{x}\in X$, the objective is
to find boundary conditions $\{\hat{f}^i\}_{i=1,2,3}$, such that the solutions $u^i,i=1,2,3$ of the equations
\begin{align*}
\begin{cases}\nabla\cdot(\gamma\nabla u^i)=0 &\textnormal{ in } X\\
u^i(x)=\hat{f}^i(x) &\textnormal{ in } \partial X
\end{cases}
\end{align*}
satisfy $\det[(\nabla u^i(\hat{x}))_{i=1,2,3}]\geq1$. 
We now have only one constraint to satisfy. It involves multiple equations and multiple boundary
conditions.
The scheme proposes to start with appropriate $\gamma_0, \{f_0^i\}_{i=1,2,3}$  (e.g. $\gamma_0\equiv1$ and $f_0^i(x_1,x_2,x_3)=x_i$)  and construct $f_s^i,i=1,2,3$ such that the solutions $u_s^i$ of
\begin{align*}
\begin{cases}\nabla\cdot(\gamma_s\nabla u_s^i)=0 &\textnormal{ in } X\\
u_s^i(x)=f_s^i(x) &\textnormal{ in } \partial X
\end{cases}
\end{align*}
and the solution $(u_s^i)'$ of
\begin{align*}
\begin{cases}\nabla\cdot(\gamma_s\nabla (u_s^i)')+\nabla\cdot(\gamma_s'\nabla u_s^i)=0 &\textnormal{ in } X\\
(u_s^i)'(x)=(f_s^i)'(x) &\textnormal{ in } \partial X
\end{cases}
\end{align*}
satisfy $\frac{d}{d s}\det[(\nabla u^i_s(\hat{x}))_{i=1,2,3}]= \sum_{i=1,2,3}\nabla (u_s^i)'(\hat{x})\cdot\Big(\nabla u_s^{i+1}(\hat{x})\times\nabla u_s^{i+2}(\hat{x})\Big) \geq0$
(we consider the expressions $i+1$ and $i+2$ modulo 3). We construct $f_s^i$ 
as the solutions of a system of ODE
\begin{align*}
\begin{cases}
\frac{\partial f_s^i}{\partial s}=F^i((f_s^i)_{i=1,2,3},s)\\
f_s^i|_{s=0}=f_0^i
\end{cases}
\end{align*}
for an appropriate $(F^i)_{i=1,2,3}$ defined below.

We construct the functionals $F^i:(C^{k+2,\alpha}(\partial X))^3\times[0,1]\to C^{k+2,\alpha}(\partial X)$
as follows. Assume $X$ is a $C^{k+3,\alpha}$ bounded domain. Let $\gamma_0,\gamma\in C^{k+n+3}(\overline{X})$
and let $\gamma_s=[(1-s)\gamma_0+s\gamma],s\in[0,1]$. Assume $0<c<\gamma_0,\gamma<C<\infty$.
For $s\in[0,1]$ and for $(f^i)_{i=1,2,3}\in (C^{k+2,\alpha}(\partial X))^3$
let $u^i$ be the $C^{k+2,\alpha}(X)$ solutions of
\begin{align*}
\begin{cases}\nabla\cdot(\gamma_s\nabla u^i)=0 &\textnormal{ in } X\\
u^i(x)=f^i(x) &\textnormal{ in } \partial X.
\end{cases}
\end{align*}
Let $v^i$ be the $C^{k+2,\alpha}(X)$ solutions of 
\begin{align*}
\begin{cases}\nabla\cdot(\gamma_s\nabla v^i)+\nabla\cdot((\gamma-\gamma_0)\nabla u^i)=0 &\textnormal{ in } X\\
v^i(x)= g^i(x) &\textnormal{ in } \partial X,
\end{cases}
\end{align*}
where $g^i\in C^{k+2,\alpha}(\partial X)$ will be prescribed below.

For this system let us consider the following auxiliary problems. Let $\lambda^i$, be the solutions of 
\begin{align*}
\begin{cases}\nabla\cdot(\gamma_s\nabla \lambda^i)=\Big(\nabla u^{i+1}(\hat{x})\times \nabla u^{i+2}(\hat{x})\Big)\cdot\nabla\delta_{\hat{x}}&\textnormal{ in } X\\
\lambda=0&\textnormal{ in } \partial X.
\end{cases}
\end{align*}
$\mbox{}$From Theorem \ref{existslambda}, $\lambda^i\in L^p(X)$ and 
$(\gamma_s \partial\lambda^i/\partial\nu)\in C^{k+2,\alpha}(\partial X)$ depend continuously on $s$.
By integration by parts and summation we obtain
\begin{align*}
 \sum_{i=1,2,3}\Big(\nabla u^{i+1}(\hat{x})\times\nabla u^{i+2}(\hat{x})\Big)\cdot\nabla v^i(\hat{x}) 
&= \sum_{i=1,2,3}\int_X\lambda^i \nabla\cdot((\gamma-\gamma_0)\nabla u^i)\\ &
- \sum_{i=1,2,3}\int_{\partial X}\gamma_s\frac{\partial \lambda^i}{\partial \nu}g^i.
\end{align*}
Let $g^i=\mu\gamma_s\partial\lambda^i/\partial\nu$, with $\mu$ chosen as 
\begin{align*}
\mu&= \sum_{i=1,2,3}\int_X\lambda^i \nabla\cdot((\gamma-\gamma_0)\nabla u^i)\Big/ \sum_{i=1,2,3}
\int_{\partial X}\Big(\gamma_s\frac{\partial \lambda_i}{\partial \nu}\Big)^2.
\end{align*}
Hence $\sum_{i=1,2,3}\nabla v^i(\hat{x})\cdot\Big(\nabla u^{i+1}(\hat{x})\times\nabla u^{i+2}(\hat{x})\Big)=0$.

By defining $F^i((f^i)_{i=1,2,3},s):=g^i$ we can prove the boundedness and continuity of
$F^i:(C^{k+2,\alpha}(\partial X))^3\times[0,1]\to C^{k+2,\alpha}(\partial X)$ as before. This yields
a Theorem analogous to Theorems \ref{secondpropF} and \ref{solofODEpoints}. An important aspect
for the argument to work is that we start with $\det[(\nabla u^i(\hat{x}))_{i=1,2,3}]\geq1$ and the evolution with $F^i$
maintains that property, hence $ \sum_{i=1,2,3}\int_{\partial X}\Big(\gamma_s\frac{\partial \lambda_i}{\partial \nu}\Big)^2$
remains uniformly bounded away from zero.

\begin{theorem}\label{solofODEdet} Assume $X$ is a $C^{k+3,\alpha}$ bounded domain.
Let $\gamma_0,\gamma\in C^{k+n+3}
(\overline{X})$ and let $\gamma_s=[(1-s)\gamma_0+s\gamma],s\in[0,1]$. Assume $0<c<\gamma_0,\gamma<C<\infty$.
Let $(f_0^i)_{i=1,2,3}$ and $\gamma_0$ be chosen appropriately (e.g. $\gamma_0\equiv1$ and $f_0^i(x_1,x_2,x_3)=x_i$).
Define $F^i:C^{k+2,\alpha}(\partial X)\times[0,1]\to C^{k+2,\alpha}
(\partial X)$  as above. Then, there exists a unique solution $s\mapsto (f_s^i)$ in $C^1([0,1];(C^{k+2,\alpha}(\partial X))^3)$
of the system of ODE
\begin{align*}
\begin{cases}
\frac{\partial}{\partial s}f_s^i=F^i((f_s^i)_{i=1,2,3},s)\\
f_s^i|_{s=0}=f_0^i.
\end{cases}
\end{align*}
For all $s\in[0,1]$, the solutions $u_s^i$ of $(P_s)$ with corresponding boundary conditions
$f_s^i$ are such that  $\det[(\nabla u^i(\hat{x}))_{i=1,2,3}]\geq1$.
\end{theorem}

This Theorem produces $\hat{f}=f_s|_{s=1}$ as the solution of the problem described 
at the beginning of this Subsection, with a condition involving finitely many equations.

\begin{remark}
\label{rem:mu}\rm In the definition of $F^i$ above, we could redefine $\mu=\min(0,\mu)$, resembling more closely the
construction presented in Definition \ref{definitionF}. Such a redefinition of $\mu$ provides a boundary condition
$\{f^i_s\}$ with $\{(f^i_s)'\}$ of minimal $L^2(\partial X)^3$ norm for each $s\in[0,1]$, among all $\{f^i_s\}$
that produce non-decreasing determinants.
\end{remark}

\begin{remark}
\label{rem:multilinear}\rm The construction presented in this Subsection works in more general settings. We may
consider $X\subset\R^n$ and replace $\det[(\nabla u^i(\hat{x}))_{i=1,2,3}]$ by
$L[(\nabla u^i(\hat{x}))_{i=1,...,m}]$ for any multi-linear function $L:(\R^n)^m\to\R$.

And if $H:(\R^n)^m\to\R$ is a continuously differentiable function with differential
$DH(z)$ uniformly bounded away from zero in the set $\{z\in (\R^n)^m: H(z)\geq 1\}$,
then the construction presented in this Subsection also works when we replace
$\det[(\nabla u^i(\hat{x}))_{i=1,2,3}]$ by $H(\nabla u^i(\hat{x})_{i=1,...,m})$.
\end{remark}

\begin{remark}
\label{rem:multipoints}\rm Extensions for conditions involving multiple equations at finitely many points can also be addressed
with this scheme.
\end{remark}

\section*{Acknowledgment} GB was partially funded by NSF grant DMS-1108608 and AFOSR Grant NSSEFF- FA9550-10-1-0194. MC was partially funded by Conicyt-Chile grant Fondecyt \#11090310.

\bibliographystyle{plain}

\end{document}